\documentclass[12pt,oneside,reqno]{amsart}
\usepackage[all]{xy}
\usepackage{amsfonts,amsmath,oldgerm,amssymb,amscd}
\UseComputerModernTips
\numberwithin{equation}{section}
\usepackage[breaklinks]{hyperref}
\allowdisplaybreaks

\usepackage{enumerate}
\usepackage{caption} 
\newtheorem{theorem}{Theorem}[section]
\newtheorem{proposition}[subsection]{\bf Proposition}

\newtheorem{lemma}[subsection]{{\bf Lemma}}

\newtheorem{remark}[subsection]{Remark}

\newcommand{\al}{\alpha}

\newcommand{\la}{\lambda}

\newcommand{\Z}{\mbox{$\mathbb Z$}}
\newcommand{\Q}{\mbox{$\mathbb Q$}}
\newcommand{\N}{\mbox{$\mathbb N$}}     

\begin{document}

	\title[Recurrence sequences and norm form equations ]{Sum of terms of recurrence sequences and $S$-units in the solution sets of norm form equations} 

\author[Darsana]{Darsana N}
\address{Darsana N, Department of Mathematics, National Institute of Technology Calicut, 
	Kozhikode-673 601, India.}
\email{darsana\_p230059ma@nitc.ac.in; darsanasrinarayan@gmail.com}

\author[Rout]{S. S. Rout}
\address{Sudhansu Sekhar Rout, Department of Mathematics, National Institute of Technology Calicut, 
	Kozhikode-673 601, India.}
\email{sudhansu@nitc.ac.in; lbs.sudhansu@gmail.com}


\begin{abstract}
	In this paper, we prove two results related to the solutions of norm form equations. Firstly, we give a finiteness result for sums of terms of linear recurrence sequences appearing in the coordinates of solutions of norm form equations. Next, we give a finiteness result concerning solutions of norm form equations representable as sums of $S$-units with a fixed number of terms. To prove these results, we use a deep results concerning the finiteness of the solutions of polynomial-exponential equations and $S$-unit equations.
\end{abstract}

\maketitle
\pagenumbering{arabic}
\pagestyle{headings}

\section{Introduction}
Let $K$ be an algebraic number field and let $\beta_1,\ldots,\beta_k\in K$ be linearly independent over $\mathbb{Q}$. Let $D\in \Z$ be the common denominator of $\beta_1,\ldots,\beta_k$. Setting $\alpha_i = D\beta_i$ for $i=1, \ldots, k$, we can say that $\alpha_1, \ldots, \alpha_k$ are algebraic integers of $K$. Let $m$ be a non-zero integer. Consider the norm form equation
\begin{equation}\label{normeq}
	N_{K/\Q}(x_1\beta_1+\cdots+x_k\beta_k)=m
\end{equation}
in integers $x_1,\ldots,x_k$. We write $X_i$ for the coordinate sets of solutions of \eqref{normeq} with $i=1,\ldots,k$. In this paper, we are interested in studying $S$-units and terms of recurrence sequences among the solutions of norm form equations. Schmidt proved that if $K$ is different from an imaginary quadratic field and  $\Q$, then \eqref{normeq} can have infinitely many solution for appropriate $m$ (see \cite{schmidt} or \cite[Chapter VII]{schmidt1980}).  Norm form equations have been studied extensively by several authors. 

One can observe that Pell type equation is a special case of norm form equation as follows. Let $d$ be a non square integer and let $1, \sqrt{d}$ be linearly independent elements of $\Q(\sqrt{d})$, then \eqref{normeq} becomes $x^2- dy^2 =m$.  There are several papers in the literature concerning terms of recurrence sequences and $S$-units occurring in the solution sets of (generalized) Pell equations (see for instance \cite{birlucatogbe2018, lucatogbe2018, egl2020, egl2021, luca 2017}). Hajdu and Sebeste\'en in \cite{hs}, studied about sums of $S$-units in the solution sets of generalized Pell equations. The same authors in \cite{hs2022} completely describe those recurrence sequences which have infinitely many terms in either of the $x, y$ coordinates of the solution sets of generalized Pell equations.

The similar type of questions are also studied in more general set up, that is, in coordinates of the solution sets of norm  form equation. In particular, Fuchs and Heintze \cite{fh2021} studied the Diophantine equation involving terms of multi-recurrences  and coordinates of solutions of norm form equations. Indeed, they proved that there are finitely many terms of the simple multi recurrences among the co-ordinates of solutions of norm form equation. In \cite{hs2024}, Hajdu and Sebeste\'en proved under certain assumptions that, an arbitrary linear recurrence sequence can have only finitely many terms in the coordinates of solutions of norm form equations.

In this paper, we study two problems in connection with the solution sets of norm form equations. At first we will study a problem related to sum of terms of recurrence sequences  in the solution sets of norm form equations. 
Then we prove a finiteness result on sums of $S$-units in the solution sets of norm form equation. This will be a generalization of the earlier result \cite[Theorem 2.1]{hs}.

\section{Notation and Main Result}
Let $a_1,\dots, a_d \in \Z$ with $a_d\neq 0$ and $U_0,\dots,U_{d-1}$ are integers not all zero and let $(U_{n})_{n \geq 0}$ be a $d$-th order linear recurrence sequence defined by
\begin{equation}\label{eq4}
	U_{n} = a_1U_{n-1} + \dots +a_dU_{n-d}.
\end{equation}
Let $\alpha_1,\dots,\alpha_q$ be the distinct roots of the corresponding characteristic polynomial  
\begin{equation}\label{eq5}
	f(x):= x^d - a_1x^{d-1}-\dots-a_d.
\end{equation} 
Then for $n\geq 0$, we have 
\begin{equation}\label{eq6}
	U_n=f_1(n)\alpha_1^n +\cdots + f_q(n)\alpha_{q}^n,
\end{equation}
where $f_i(n)$ are nonzero polynomials with degree less than the multiplicity of $\alpha_i$; the coefficients of $f_i(n)$ are elements of the field $\Q( \alpha_1, \ldots, \alpha_q)$ (see \cite[Theorem C.1]{tnshorey}). 

The sequence $(U_n)_{n \geq 0}$ is called {\it degenerate} if there are integers $i, j$ with $1\leq i< j\leq q$ such that $\alpha_i/\alpha_j$ is a root of unity; otherwise it is called {\it non-degenerate}. The sequence $(U_n)_{n \geq 0}$ is called {\it simple} if $q=d$. In this case, \eqref{eq6} becomes
\begin{equation}\label{eqa}
	U_n=f_1\alpha_1^n +\cdots + f_d\alpha_{d}^n,
\end{equation}
where $\eta_i$'s are constants. There is a natural generalization of linear recurrence sequences, which is known as multi-recurrences. It is given as
\begin{equation}\label{mrec}
	F(n_1,\ldots,n_s)=\sum_{i=1}^{r}P_{i}(n_1,\ldots,n_s)\alpha_{i1}^{n_1}\ldots\alpha_{is}^{n_s}
\end{equation}
where $s$ and $r$ are positive integers, $P_1,\ldots, P_r$ are polynomials in $s$ variables defined over $K$, and $n_1,\ldots,n_s$ are non-negative integers. Such polynomial-exponential functions $F:\mathbb{N}_{0}^{s}\rightarrow K$ are called multi recurrences, where $\mathbb{N}_{0}$ is the set of non-negative integers. We say that $F$ is defined over a field $K$ if the coefficients and the bases $\alpha_{i1},\ldots, \alpha_{is}$ for $i=1,\ldots,r$ are in $K$. If $F$ is defined over $K$, then it takes values in $K$. 
\par 
Consider the number field $\mathbb{Q}(\alpha_{11},\ldots,\alpha_{1s},\ldots,\alpha_{r1},\ldots,\alpha_{rs})$. We shall use the abbreviation $\boldsymbol{\alpha}_i^{{\bf n}} = \alpha_{i1}^{n_1}\cdots \alpha_{is}^{n_s}$ and hence \[P_{i}(\textbf{n})\boldsymbol{\alpha}_{i}^{\textbf{n}}= P_{i}(n_1,\ldots,n_s)\alpha_{i1}^{n_1}\ldots\alpha_{is}^{n_s}.\]
We say that $F$ is {\it degenerate} if there are integers $i, j$ with $1\leq i< j\leq r$ such that $\boldsymbol{\alpha}_{i}^{\textbf{n}}=\boldsymbol{\alpha}_{j}^{\textbf{n}}$ for some $\textbf{n}\in \mathbb{N}_0^{s}$; otherwise it is called {\it non-degenerate}.
We define $\boldsymbol{\alpha}_{i}$ is irrational if $\alpha_{ij}$ are irrational for all $j=1,\ldots, s$. 

Now we give our first result about the the Diophantine equation
 \begin{equation}\label{eq6a}
		U_{n_1} + U_{n_2} \in X_1 \cup \cdots\cup X_k
	\end{equation} 
 in $(n_1, n_2)\in \N_{0}^2$.
 
\begin{theorem}\label{thm2}
	Let $K$ be an algebraic number field of degree $k$ and $(U_n)$ be a  non-degenerate linear recurrence sequence
	of integers of order $d$ as in \eqref{eq6}. Assume that the roots of the characteristic polynomial $\alpha_i, (1\leq i\leq q)$ are pairwise multiplicatively independent and none of the root is a root of unity. 
If $a_d\neq \pm 1$ in \eqref{eq4}, then \eqref{eq6a} holds only for finitely many tuples $(n_1, n_2)$. Further, the number of such tuples $(n_1, n_2)$ is bounded by $C_1 = C_1(m, k, d)$, where $C_1$ is an effectively computable constant depending only on $m, k$ and, $ d$.
\end{theorem}

Let $L$ be an algebraic  number field, and let $S =\{\mathcal{P}_1,\ldots, \mathcal{P}_l\}$ be a finite set of prime ideals in $L$. Denote $U_{S}$ for the set of $S$-units in $L$, that is, for the set of $w\in L$ for which the principal fractional ideal $(w)$ can be represented as
\[(w) = \mathcal{P}_{1}^{b_1}\cdots \mathcal{P}_{l}^{b_l}\quad (b_1, \ldots, b_l\in \Z).\]

In next result, we study about sums of $S$-units in the solution sets of norm form equations. In particular, we study the Diophantine equation
 \begin{equation}\label{snormeq}
	w_1+\cdots+w_t \in X_1 \cup \cdots \cup X_k
\end{equation}
in $(w_1,\ldots,w_t)\in U_{S}^{t}$, for some $t\in \mathbb{Z}_{\geq 0}$.
 \begin{theorem}\label{thm1}
 	Let $U_S$ be the set of $S$-units as defined above
  and let $t\geq 1$. If 
\begin{equation}\label{nonzerow}
	w_{i_1}+\cdots+w_{i_j}\neq 0
\end{equation}
for any $0<j\leq t$ and $1\leq i_1<\cdots<i_j\leq t$, then \eqref{snormeq} holds for finitely many tuples $(w_1,\ldots,w_t)\in U_{S}^{t}$. In fact, the number of such tuples is bounded by $C_2$,  which is an effectively computable constant depending only on $m, t$ and $l$.
 \end{theorem}
 
Below we have shown that there are infinitely many indices $(n_1, n_2)$ of \eqref{eq6a} if any of the assumptions in Theorem \ref{thm2} is not satisfied.  
\begin{remark}
Let $K=\mathbb{Q}(\sqrt{13})$ and consider the norm form equation 
\begin{equation}\label{eqrem1}
	N_{K/\Q}(x_1+\sqrt{13}x_2)=4.
\end{equation}
 Then $$X_1=\{11,119,1298,14159,154451,1684802,\ldots
\}$$ and $$X_2=\{3,33,360,3927,42837,467280,\ldots\}.$$ It is well known that every $3^{rd}$ multiple term  in the sets $X_1$ and $X_2$ are even numbers. Now consider the sequence $(U_n)$ defined by $U_0=0,U_1=2$ and $U_{n+2}=2U_{n+1}-U_n(n\geq 0)$. Thus, $(U_n)$ is exactly the sequence of even numbers, so that $U_{n_1}+U_{n_2}\in X_1\cup X_2$ for infinitely many $(n_1,n_2)$. Here, observe that $f(x)=x^2-2x+1$ is the characteristic polynomial of $(U_n)$ whose roots are root of unity.
\end{remark}

\begin{remark}
Let $\beta=\sqrt{5}$ and $K=\mathbb{Q}(\beta)$. Then $1,\sqrt{5}$ is an integral basis of $K$. Consider the norm form equation 
\begin{equation}\label{eqrem2}
	N_{K/\Q}(x_1+x_2\sqrt{5})=4.
\end{equation}
One can easily obtain the  solution for $x_1$ is $3,7,18,\ldots$. Let $U_0=0, U_1=3$, and $U_{n+2}=U_{n+1}-U_{n}$. Then the terms of the sequence $(U_n)$ are given by $0, 3, 3, 0, -3, -3, 0, 3, \ldots$. This shows that $U_{n_1}+U_{n_2}=3\in X_1$ for infinitely many $(0, n_2)$ or  $(n_1, 0)$. In this case, the roots of characteristic polynomial of $(U_n)$ are $\frac{1+\sqrt{3}i}{2}  and  \frac{1-\sqrt{3}i}{2}$ whose ratio is a cube root of unity. So non-degeneracy is a necessary condition for the theorem.
\end{remark}

\begin{remark}
Let $K=\mathbb{Q}(\sqrt{2})$ and consider  
	\begin{equation}\label{eqrem3}
	N_{K/\Q}(x_1+x_2\sqrt{2})=-1.
\end{equation}
Then $1, \sqrt{2}$ is an integral basis of $K$. Further, $3+\sqrt{2}$ is the fundamental unit of $K$. So we can write 
	\begin{equation}\label{eq2.9}
		x_1+x_2\sqrt{2}=(1+\sqrt{2})(3+2\sqrt{2})^k,\qquad \text{for some}\,  k\in \mathbb{Z}
	\end{equation}
	Let $\tau_1$ and $\tau_2$ be the isomorphisms on $K$. Then taking conjugates in (\ref{eq2.9}), we get \\
$$	\begin{pmatrix}
		1&\sqrt{2}\\
		1& -\sqrt{2}
	\end{pmatrix}\begin{pmatrix}
	x_1\\
	x_2
	\end{pmatrix}=\begin{pmatrix}
	(1+\sqrt{2})(3+2\sqrt{2})^k\\
		(1-\sqrt{2})(3-2\sqrt{2})^k
	\end{pmatrix}.$$\\[.3cm]
From this we infer \[X_2=\frac{(1+\sqrt{2})^{2k+1}-(1-\sqrt{2})^{2k+1}}{2\sqrt{2}}.\] Now define a sequence $(U_n)$ by $U_0=0, U_1=1$ and $U_{n+1}=6U_{n}-U_{n-1}(n\geq1)$ and the Binet form is \[U_n=\frac{(3+2\sqrt{2})^n-(3-2\sqrt{2})^n}{4\sqrt{2}}.\] Also we see that $U_n+U_{n+1}\in X_2$ for every $n\geq 0$. Note that the roots $3+2\sqrt{2}$ and $3-2\sqrt{2}$ are multiplicatively dependent.
\end{remark}

\begin{remark}
Suppose 
$U_0=0,U_1=1$ and $U_{n+2}=10U_{n+1}-U_n$. The characteristic polynomial of $(U_n)$ is $x^2-10x+1$ and observe that the constant term $a_2=1$. 

Let $K=\mathbb{Q}(\sqrt{2}+\sqrt{3})$ and $\beta=\sqrt{2}+\sqrt{3}$ be a fundamental unit of $K$.  Let $\{v_1,v_2,v_3,v_4\}$ be a basis as constructed in \cite[Theorem 1.3]{elisa} and consider
$$N_{K/\Q}(x_1v_1+\cdots+x_4v_4)=1,$$ that is, 
we can write  $x_1(k)v_1+\cdots+x_4(k)v_4=\eta\beta^k$ where $\eta\in K$.
Then for any integer $k\geq 2$, we have 
$X_1(k)=U_{n+1}+U_n$ if $k=2n+1$ (see \cite[Corollary 4.2]{elisa}).  Hence (\ref{eq6a}) has infinitely many solutions. 
\end{remark}

\section{Auxiliary results}
In this section, we give necessary background to prove our results. First consider the equation
\begin{equation}\label{smith}
	\sum_{l=1}^{r}P_l(\bf{x})\mathfrak{g}_{\it{l}}^{\bf{x}} = 0
\end{equation}
in variables $\bf{x} = (x_1, . . . , x_s) \in \Z^s,$ where the $P_l$ are not identically zero polynomials in $s$ variables with coefficients in an algebraic number field $K$, and\\
$$\mathfrak{g}_l^{\bf{x}} =g^{x_1}_{l_1} \cdots g^{x_s}_{l_s}$$
with given nonzero $g_{l_1}, \ldots , g_{l_s} \in K
\quad (l = 1, \ldots , r).$
Let $\pi$ be a partition of the set $\{1, \ldots , r\}$. The sets $\lambda\subset \{1, \ldots, r\}$ belonging to $\pi$ will be considered to be elements of $\pi$: we write $\lambda \in \pi$. Given a partition $\pi$, the set of equations
\begin{equation}\label{ref}
	\sum_{l\in \la}P_l(\bf{x})\mathfrak{g}_{\it{l}}^{\bf{x}} = 0, \quad \la \in \pi
\end{equation}
yields a refinement of \eqref{smith}. When $\pi'$ is a refinement of $\pi$, then system of equations corresponding to the partition $\pi'$ implies system of equations corresponding to the partition $\pi$. Let $S(\pi)$ be the set of solutions of \eqref{ref} which are not solutions of \eqref{ref} with any proper refinement $\pi'$ of $\pi$. 

Set $i \overset{\pi}{\sim} j$ if $i$ and $j$ lie in the same subset $\la$ belonging to $\pi$. Let $G(\pi)$ be the
subgroup of $\Z^s$ consisting of $\bf{z}$ with
\[\mathfrak{g}_i^{\bf{z}} = \mathfrak{g}_j^{\bf{z}}\] 
for any $i, j$ with $i \overset{\pi}{\sim} j$.

\begin{lemma}\label{gp}
	Using the above notation, if $G(\pi) = \{0\}$ then we have
	\[S(\pi)<2^{35A^3}D^{6A^2}\]
	with $D = \deg(K)$ and
	\[ A = \max \left(s, \sum_{l\in \Lambda}\binom{s+\delta_l}{s}\right),\]
	where $\delta_l$ is the total degree of the polynomial $P_l.$
\end{lemma}

\begin{proof}
	The statement is Theorem 1 in \cite{ss2000}.
\end{proof}

Following the arguments from \cite{arith} and \cite{fh2021}, we prove the following.
\begin{proposition}\label{lem normeqsol}
If \eqref{normeq} has a  solution, then all solutions are given by
	\begin{equation*}
		x_{i}^{(j)}=F_{i}^{(j)}(\textbf{n}) 
	\end{equation*}
	for $j=1,\ldots, J$ and $i=1,\ldots, k$ with some multi-recurrence
	$$F_{i}^{j}(\textbf{n})=\sum_{h=1}^{r} P_{i_h}^{(j)}(\textbf{n})\boldsymbol{\alpha}_{i_h}^{{\textbf{n}}},$$
	where $\boldsymbol{\alpha}_{i_h}$ are irrational as well as all co-ordinates are algebraic over $\mathbb{Q}$
and $J$ is some positive integer such that $J<C_3$, where $C_3$ is an effectively computable constant depending on $k$ and $m$.
\end{proposition}
\begin{proof}
	We know that there are only finitely many pairwise non-associate algebraic integers $\mu$ in $K$ of norm $m$, and let the number of such elements be $J$. So $J$ can be bounded in terms of $k$ and $m$ (see \cite[Lemma 5]{gyory}). Let $\mu_1,\ldots,\mu_J$ be the pairwise non-associate algebraic integers  in $K$ of norm $m$.  If $(x_1^{(j)},\ldots,x_{k}^{(j)})\in \mathbb{Z}^{k}$ is a solution of \eqref{normeq}, then we have  
	\begin{equation}\label{lem eq3.4}
	x_{1}^{(j)}\beta_1+\cdots x_{k}^{(j)}\beta_k= \epsilon \mu_j	
	\end{equation} 
	for some $j\in J$ and $\epsilon$ is a unit in $K$.  Let $\epsilon_1,\ldots,\epsilon_s$ be a system of fundamental units in $K$. Then \eqref{lem eq3.4} yields 
	\begin{equation}\label{eq x}
			x_{1}^{(j)}\beta_1+\cdots +x_{k}^{(j)}\beta_k= \gamma\epsilon_{1}^{a_1}\cdots \epsilon_{s}^{a_s}\mu_j
	\end{equation}
	where $\gamma$ is a root of unity in $K$. Since the number of roots of unity in $K$ is bounded in terms of $k$, we may assume that this $\gamma$ is fixed. Now taking the conjugates of (\ref{eq x}), we get
	\begin{equation}\label{lem matrix}
		\begin{pmatrix}
			\tau_1(\beta_1)&\cdots&\tau_1(\beta_k)\\
			\vdots&\ddots&\vdots\\
			\tau_k(\beta_1)&\cdots&\tau_k(\beta_k)
		\end{pmatrix} 
		\begin{pmatrix}
			x_1^{(j)}\\
			\vdots\\
			x_k^{(j)}
		\end{pmatrix}= \begin{pmatrix}
			\tau_1(\gamma\mu_j)u_{11}^{a_1}\cdots u_{1s}^{a_s}\\
			\vdots\\
			\tau_k(\gamma\mu_j)u_{k1}^{a_1}\cdots u_{ks}^{a_s}
		\end{pmatrix}
	\end{equation}
	where $\tau_1,\ldots,\tau_k$ are the isomorphism of $K$ into $\mathbb{C}$ (in any order), and $u_{ih}=\tau_i(\epsilon_h)$ for $i=1,\ldots, k$ and $h=1,\ldots, s$. The determinant of the matrix on the left hand side of (\ref{lem matrix}) is non-zero, so we write 
		\begin{equation}\label{lem eq3.8}
			x_i^{(j)}=c_{1i}^{(j)}u_{11}^{a_1}\cdots u_{1s}^{a_s}+\cdots+c_{ki}^{(j)}u_{k1}^{a_1}\cdots u_{ks}^{a_s}
		\end{equation}
		where $c_{ih}^{(j)}$ belong to the normal closure of $K$ for $i=1,\ldots,k$ and $u_{ih}(1\leq i\leq k, 1\leq h\leq s )$ are irrational. For if, any of the $u_{ij}$ is rational, then $\epsilon_j$ must be rational. Since $\epsilon_j$ is a fundamental unit, $N_{K/\Q}(\epsilon_j)$ should be $\pm 1$ which is not possible as we obtained that $\epsilon_j$ is rational and  considered $K$ is different from imaginary quadratic field and $\mathbb{Q}$. Thus, (\ref{lem eq3.8}) is the required multi-recurrence. 
\end{proof}

\begin{remark}
Suppose for some distinct indices $i, j$ such that  $\boldsymbol{\alpha}_{i}^{\bf{n}} = \boldsymbol{\alpha}_{j}^{\bf{n}}$, that is, 
\begin{equation}\label{eqnondeg}
	\alpha_{i1}^{n_1}\cdots\alpha_{is}^{n_s}=\alpha_{j1}^{n_1}\cdots\alpha_{js}^{n_s}
\end{equation}
holds for infinitely many vectors  $(n_1, \ldots, n_s)$. Then we can construct a multi-recurrence 
\begin{equation*}
	F'(n_1,\ldots,n_s)=\sum_{i=1}^r(P_i\alpha_{i1}^{n_1}\cdots\alpha_{is}^{n_s}-P_i\alpha_{j1}^{n_1}\cdots\alpha_{js}^{n_s})
\end{equation*}
which is zero for the vectors in our sequence such that 
$$\begin{array}{cl}
	F_0(n_1,\ldots,n_s)&:=F(n_1,\ldots,n_s)-	F'(n_1,\ldots,n_s)\\
	\\
	&= \sum_{j=1}^{\tilde{r}}\tilde{P_{j}}\alpha_{j1}^{n_1}\cdots\alpha_{js}^{n_s}
\end{array}$$
does not contain two summands which satisfy a relation of the shape \eqref{eqnondeg}.
\end{remark}

Note that both effective and ineffective results related to sums of $S$-units in linear recurrence sequence has been studied in \cite{behaj}.  We need a similar type of ineffective result for multi-recurrences.
\begin{proposition}\label{prop1}
	Let $F(\textbf{n})$ be a non-degenerate multi-recurrence defined as in (\ref{mrec}) and $w_1,\ldots, w_t \in U_S$. Further, assume that $\alpha_{ij}$ are algebraic and  $\boldsymbol{\alpha}_{i}$ is irrational for some $i=1,\ldots, r$. Then for any fixed $t\geq 1$,
	\begin{equation}\label{prop eq}
		F(\textbf{n})=w_1+\cdots+w_t	
	\end{equation} 
is solvable at most for finitely many $\textbf{n}$, which can be bounded by an effectively computable constant $C_4$ depending on $r$ and $t$.
\end{proposition}

\begin{proof}
We rewrite \eqref{prop eq} using \eqref{mrec}, as  
	\begin{equation}\label{prop 3.4}
		\sum_{i=1}^{r}P_{i}(n_1,\ldots, n_s)\alpha_{i1}^{n_1}\ldots\alpha_{is}^{n_s}= \mathcal{P}_1^{b_{11}}\cdots \mathcal{P}_l^{b_{1l}}+\cdots+ \mathcal{P}_1^{b_{t1}}\cdots \mathcal{P}_l^{b_{tl}}
	\end{equation}
	where $b_{uv}(1\leq u\leq t, 1 \leq v \leq l)$ are non-negative integers. Let $\pi$ be any partition of the set $\{1,\ldots, r, r+1, \ldots, r+t\}$. Consider first those solutions of (\ref{prop 3.4})  which do not satisfy any refinement of $\pi$, and for which there exists $i, j$ with $1\leq i<j\leq r$ belong to the same $\lambda \in \pi$. Since $F$ is non-degenerate, we conclude that $G(\pi)=\{\textbf{0}\}$. The boundedness of the number of such solutions is followed by Lemma \ref{gp}. 
	
Now consider those solutions of (\ref{prop 3.4}) which still do not satisfy any refinement of $\pi$, and for which the indices $i$ with $1\leq i\leq r$ belong to different classes of $\pi$. Consider an $i\;(1\leq i\leq r)$ such that $\boldsymbol{\alpha}_i$ is irrational. Take a class $\lambda\in \pi$ which contains $i$. There exists an isomorphism $\sigma$ of $\mathbb{Q}(\alpha_{11},\ldots,\alpha_{1s},\ldots,\alpha_{r1},\ldots\alpha_{rs})$ into $\mathbb{C}$ such that $\sigma(\boldsymbol{\alpha}_{i})=\boldsymbol{\alpha}_{j}$ with some $1\leq i\neq j\leq r$. So we have 
\[P_i(\textbf{n})\boldsymbol{\alpha}_{i}^{\textbf{n}}=w_{r_1}+\cdots+w_{r_h}\] for some $\{r_1,\ldots, r_h\}\subset\{1,\ldots,t\}$, this implies 
\[P_i(\textbf{n})\boldsymbol{\alpha}_{i}^{\textbf{n}}=\sigma(P_i(\textbf{n}))\boldsymbol{\alpha}_{j}^{\textbf{n}}.\] By Lemma \ref{gp} this equation has only finitely many solutions, and their number can be effectively bounded. Since the number of partitions $\pi$ of $\{1,\ldots, r+t\}$ is bounded in terms of $r$ and $t$, this completes the proof. 
\end{proof}

\begin{remark}
	In (\ref{mrec}), let us take $s=2,r=2$, $P_1=P_2=1$, and $\boldsymbol{\alpha_1}=(5,6),\ \boldsymbol{\alpha_2}=(10,3)$ and $S=\{2,3,5\}$. Then \eqref{snormeq} becomes
	\begin{equation}\label{remeq1}
 2^{b_{11}}3^{b_{12}}5^{b_{13}}+ 2^{b_{21}}3^{b_{22}}5^{b_{23}}=	5^{n_1}\cdot 6^{n_2}+ 10^{n_1}\cdot 3^{n_2}.
	\end{equation}
So that \eqref{remeq1} has infinitely many solutions in $(n_1, n_2)$ with $n_1=b_{13}=b_{21}=b_{23}$ and $n_2=b_{11}=b_{12}=b_{22}$. Observe that $\boldsymbol{\alpha_1}$ and $\boldsymbol{\alpha_2}$ are not irrational. 
\end{remark}

\begin{proposition}\label{lem uniteq}
Let $b_1,\ldots,b_t$ be non-zero elements of the number field $L$. Then the equation
\begin{equation}\label{lem suniteq}
	b_1y_1+\cdots+b_ty_t=1 
\end{equation}
has atmost $C_5$ solutions $(y_1,\ldots,y_t)\in U_{S}^{t}$ for which the left hand side of (\ref{lem suniteq}) has no vanishing subsums. Here $C_5$ is an effectively computable constant depending only on $t, l$ and $\deg L$.	
\end{proposition}
\begin{proof}
	See \cite[Theorem 6.1.3]{eversgyo} and \cite[Corollary 4.1.5]{eversgyo}.
\end{proof}

We also need the following result.

\begin{lemma}\label{lem3.4}
Let $(U_n)_{n\geq 0}$ be a linear recurrence sequence of integers of order at least two. Further, assume that the roots of the characteristic polynomial $\alpha_i, (1\leq i \leq q)$ are pairwise multiplicatively independent. Let $T$ be the set of places of $\Q(\alpha_1, \ldots, \alpha_q)$ such that $||\alpha_i||_v=1$ for all $i=1, \ldots, q$ and $v\not \in T$.  Suppose $\Delta \subseteq \{1, \ldots, q, 1, \ldots, q\}$. Then, there are finitely many tuples $(n_1, n_2)\in \Z^2_{\geq0}$ satisfying
  \begin{equation}\label{eq1lem5}
  \sum_{i\in \Delta}f_i(n_1)\alpha_i^{n_1}+f_i(n_2)\alpha_i^{n_2} =0
  \end{equation}
and \begin{equation}\label{neqeq155}
f_i(n_1)\alpha_{i}^{n_1}+f_i(n_2) \alpha_i^{n_{2}}\neq 0, \quad (1\leq i \leq q).
\end{equation}
\end{lemma}
\begin{proof}
See Theorem 1 in \cite{pt}.
\end{proof}
Now we are ready to prove our main results.
\section{Proof of Theorem \ref{thm2}}
Let $(U_n)$ be a linear recurrence sequence of order
$d$ as given in \eqref{eq6}, satisfying the assumptions of the statement. Let $T$ be the minimal set of places of $\Q(\alpha_1, \ldots, \alpha_q)$ such that $||\alpha_i||_v=1$ for all $i=1, \ldots, q$ and $v\not \in T$. Let $(x_1,\ldots,x_k)\in \mathbb{Z}^{k}$ be the solution of \eqref{normeq}. We need to bound the indices $(n_1,n_2)$ such that $U_{n_{1}}+U_{n_{2}}=x_i$ for some $i$ with $1\leq i\leq k$. By using the Proposition \ref{lem normeqsol} and \eqref{eq6}, the relation (\ref{eq6a}) becomes
\begin{equation}\label{2.5}
	c_{1i}\textbf{u}_{1}^{\textbf{a}}+\cdots+	c_{ki}\textbf{u}_{k}^{\textbf{a}}=f_1(n_1)\alpha_1^{n_1} +\cdots + f_q(n_1)\alpha_{q}^{n_1}+f_1(n_2)\alpha_1^{n_2} +\cdots + f_q(n_2)\alpha_{q}^{n_2},
\end{equation}
where $\textbf{a}=(a_1,\ldots,a_s)$ and $\textbf{u}_{i}=(u_{i1},\ldots,u_{is})$ for some $i=1,\ldots,k$. Note that by our convention on the minimality on $d$ in \eqref{eq4}, here none of $f_i(n_1),f_i(n_2)$ and $\alpha_i$ are zero. The $2q$ summands on the right hand side of (\ref{2.5}) will be parameterized by the symbols
$\overline{1},\ldots,\overline{q},\hat{1},\ldots,\hat{q}$. We may assume that $i$ is fixed and define the $(r+2)$-tuples $\mathfrak{g}_j$ as\\
$$ \mathfrak{g}_j= \left\{\
	 \begin{array}{lc}
(u_{j1},\ldots,u_{js},1,1),& 1\leq j\leq k,\\
(1,\ldots,1,\alpha_j,1),& \overline{1}\leq j\leq \overline{q},\\
(1,\ldots,1,1,\alpha_j),& \hat{1}\leq j\leq \hat{q}\\
\end{array}
\right.$$
and $$P_j(\textbf{y})= \left\{\
\begin{array}{lc}
	c_{ji},& 1\leq j\leq k,\\
	-f_j(y_{r+1}),& \overline{1}\leq j\leq \overline{q}\\
	-f_j(y_{r+2}),& \hat{1}\leq j\leq \hat{q}\\
\end{array}\right.$$
where $P_j\, (1\leq j\leq k, \bar{1}\leq j\leq \bar{q}, \hat{1}\leq j\leq \hat{q})$ is a polynomial in $r+2$ variables $\textbf{y}=(y_1,\ldots,y_{r+2})$. Then \eqref{2.5} can be rewritten as
\begin{equation}\label{2.55}
\sum_{j=1}^{k+2q}P_j(\textbf{y})\mathfrak{g}_j^{\textbf{y}}=0.
\end{equation}
Let $J=\{\ 1,\ldots,k,\overline{1},\ldots,\overline{q},\hat{1},\ldots,\hat{q}\}$ and $\pi$ be a partition of $J$. Consider the refinement of \eqref{2.55}
\begin{equation}\label{2.6}
\sum_{j\in \Delta}P_j(\textbf{y})\mathfrak{g}_j^{\textbf{y}}=0,\qquad(\Delta \in \pi).
\end{equation}
We are interested with the solutions of \eqref{2.6} which are not solutions of proper refinement of $\pi$. That is, we consider the solutions of  \eqref{2.6} in $S(\pi)$.
\par

    Assume first that there is a subset $\Delta$ of  $\pi$ such that $\Delta \subseteq \{\bar{1}, \ldots , \bar{q}\}$. If $\Delta$ is singleton set, that is, $\Delta = \{i\}$,  and $\pi$ yields $f_i(n_1)\alpha_i^{n_1} =0$, which is not possible as $\al_i \neq 0, \; (1 \leq i \leq q)$. So, $|\Delta| =1$ is not possible. Thus $|\Delta| \geq 2$. If  we had $i \overset{\pi}{\sim} j$ with some $i\neq j$ then this yields $\alpha_i^{n_1}  = -\alpha_j^{n_1}$.  Since the recurrence sequence $(U_n)_{n\geq 0}$ is non-degenerate, we have $G(\pi) = \{\bf{0}\}$.  Hence, by Lemma \ref{gp} we get an upper bound for the number of these values of $n_1$ in terms of $d$.

Similarly, if we assume there is a subset $\Delta$ of  $\pi$ such that $\Delta \subseteq \{\hat{1}, \ldots, \hat{q}\}$, then by proceeding as above  we get an upper bound for the number of these values of $n_2$ in terms of $d$.

Now assume that $\Delta \subseteq \{\bar{1}, \ldots, \bar{q}, \hat{1}, \ldots, \hat{q}\}$ such that $i \in \{\bar{1}, \ldots, \bar{q}\}$ and  $j\in \{\hat{1}, \ldots, \hat{q}\}$. Suppose $\Delta = \{i, j\}$ with $i\neq j$. In this case we will get an equation of the form
\begin{equation}
f_{i}(n_1)\alpha_{i}^{n_1} +f_{j}(n_2)\alpha_{j}^{n_2}=0.
\end{equation}
Since by our assumption, $\alpha_i$'s are pair wise multiplicatively independent, we get $G(\pi) = \{\bf{0}\}$. Hence, by Lemma \ref{gp} we get an upper bound for the number of these values of $n_1, n_2$ in terms of $d$. Next suppose that $\Delta = \{i, i\}$. Further, if $\al_{i}^{n_1} + \al_{i}^{n_2} \neq 0$, then by Lemma \ref{lem3.4}, we get an upper bound for the number of these values of $n_1, n_2$ in terms of $d$. On the other hand if $\al_{i}^{n_1} + \al_{i}^{n_2} = 0$, then $\al_i$ will be a root of unity, which is a contradiction.

Now we consider $\Delta\in \pi$ which contains an index from $\{\overline{1},\ldots,\overline{q},\hat
1,\ldots \hat{q} \}$ and assume that there is a $j_2\in \Delta$ with $1\leq j_2\leq k$ such that $P_{j_{2}}$ is not identically zero. Suppose $a_d\neq \pm 1$. There arises two cases: 
\begin{enumerate}
	\item[(I)] If $j_1,j_3\in \Delta$ such that $\overline{1}\leq j_1\leq \overline{q}$ and $\hat{1}\leq j_3\leq \hat{q}$. If $\textbf{y}\in G(\pi)$, then as $j_1\overset{\pi}{\sim}j_2$ and $j_3\overset{\pi}{\sim}j_2$, we have
	\begin{equation}\label{2.7}
		 u_{j_{2}1}^{y_1}\cdots u_{j_{2}r}^{y_r}=\alpha_{j_1}^{y_{r+1}} \qquad \text{and}\qquad u_{j_{2}1}^{y_1}\cdots u_{j_{2}r}^{y_r}=\alpha_{j_3}^{y_{r+2}}
	\end{equation} respectively.
	Since $\alpha_{j_1}$ and $\alpha_{j_3}$  is not a unit in $\tau_{j_2}(K)$, we get $y_{r+1}=0$ and $y_{r+2}=0$. Then, as $u_{j_{2}1}^{y_1}\cdots u_{j_{2}r}^{y_r}$ is a system of fundamental units in $\tau_{j_2}(K)$, we obtain $y_1=\cdots=y_r=0$. Hence $\textbf{y}=\bf 0$, so $G(\pi)=\{{\bf 0}\}.$
\item [(II)] Suppose $j_1\in \Delta$ with $\overline{1}\leq j_1\leq \overline{q}$ and no other index in $\{ \hat{1},\ldots,\hat{d}\}$ belongs to $\Delta$.  If $\textbf{y}\in G(\pi)$, then as $j_1\overset{\pi}{\sim}j_2$, we have $u_{j_{21}}^{y_1}\cdots u_{j_{2r}}^{y_r}=\alpha_{j_1}^{y_{r+1}}$. Since $\alpha_{j_1}$  is not a unit in $\tau_{j_2}(K)$ we get $y_{r+1}=0$. Then, as $u_{j_{2}1}^{y_1}\cdots u_{j_{2}r}^{y_r}$ is a system of fundamental units in $\tau_{j_2}(K)$, we obtain $y_1=\cdots=y_r=0$. Hence $\textbf{y}=\bf 0$, so $G(\pi)=\{{\bf 0}\}.$
\end{enumerate}
Thus, by Lemma \ref{gp} our statement follows.

\section{Proof of Theorem \ref{thm1}} 
Let $w_1,\ldots, w_t\in U_S$ satisfy (\ref{snormeq}) and (\ref{nonzerow}). Assume  that $$w_1+\cdots +w_t=x_i$$ for some $x_i\in X_i (i=1,\ldots,k)$. By Proposition \ref{lem normeqsol}, we have 
\begin{equation}\label{thm eq1}
	w_1+\cdots+w_t= F_{i}^{(j)}(\textbf{n})
\end{equation}
for some $j\in\{1,\ldots,J\}$ and $\textbf{n}\in \mathbb{N}_0^{s}$, where $J$ is bounded in terms of $k$ and $m$ and  $F_{i}^{(j)}(\textbf{n})$ is a term of a non-degenerate multi-recurrence $F_{i}^{(j)}$. So by Proposition \ref{prop1}, the number of such $\textbf{n}$ satisfying (\ref{thm eq1}) is $C_4$. Further by \eqref{nonzerow}, $F_{i}^{(j)}(\textbf{n})\neq 0$. Setting $b_q=1/F_{i}^{(j)}(\textbf{n})$ for $1\leq q\leq t$, \eqref{thm eq1} can be written as
\begin{equation}\label{eq5.5}
	b_1w_1+\cdots+b_tw_t=1.
\end{equation}
Then from \eqref{nonzerow}  and Proposition \ref{lem uniteq}, there are finitely many tuples $(w_1, \ldots, w_t)\in U_{S}^t$ satisfying  \eqref{eq5.5}. Thus, the result follows as the number of the above type equations appearing is at most $C_5$.


{\bf Acknowledgment:}  Both author's work is supported by a grant from Science and Engineering Research Board (SERB)(File No.:CRG/2022/000268) and S.S.R. is supported by a grant from National Board for Higher Mathematics (NBHM), Sanction Order No: 14053.

\end{document}